\documentclass[11pt,twoside,a4paper]{article}
\usepackage[english]{babel}   
\usepackage{amsmath}
\usepackage{amssymb}
\usepackage{amsthm}
\usepackage{enumerate}
\usepackage{eucal}
\usepackage{rotating}
\usepackage{psfrag}
\usepackage{epsfig}
\usepackage{makeidx}
\usepackage{bbm}
\usepackage{framed}
\usepackage{graphicx}

\theoremstyle{definition}
\newtheorem{definition}{Definition}[section]
\newtheorem{theorem}[definition]{Theorem}
\newtheorem{proposition}[definition]{Proposition}
\newtheorem{lemma}[definition]{Lemma}
\newtheorem{corollary}[definition]{Corollary}
\newtheorem*{acknowledgement*}{Acknowledgement}
\theoremstyle{remark}
\newtheorem{remark}[definition]{Remark}
\newtheorem{example}[definition]{Example}

\newcommand{\N}{\ensuremath{\mathbb{N}}}

\newcommand{\hide}[1]{}

\makeindex

\title{Disjointness preserving $\mathrm{C}_0$-semigroups and local operators on ordered Banach spaces}
\author{Anke Kalauch, Onno van Gaans and Feng Zhang}

\date{}

\begin{document}

\maketitle

\abstract{We generalize results concerning $\mathrm{C}_0$-semigroups on Banach lattices to a setting of ordered Banach spaces.  We prove that the generator of a disjointness preserving $\mathrm{C}_0$-semigroup is local. Some basic properties of local operators are also given.
We investigate cases where local operators generate local $\mathrm{C}_0$-semigroups, by using Taylor series or Yosida approximations. As norms we consider regular norms and show that bands are closed with respect to such norms. Our proofs rely on the theory of embedding pre-Riesz spaces in vector lattices and on corresponding extensions of regular norms.}



\section{Introduction}

In the theory of $\mathrm{C}_0$-semigroups, many results involve the order structure of the underlying Banach space. For instance, a rich theory of disjointness preserving semigroups and semigroups
with local generators has been developed in \cite{Are1986C, Are1986B}. The Banach spaces in these works are Banach lattices. However, there are many Banach spaces that have a natural partial order but that are not vector lattices. In the present paper, we investigate how results on disjointness
preserving $\mathrm{C}_0$-semigroups on Banach lattices can be generalized to the more general setting of ordered Banach spaces. 

Ordered Banach spaces are covered by the theory of pre-Riesz spaces. Pre-Riesz spaces are those partially ordered vector spaces that have a suitable vector lattice completion, see Section \ref{sec.preli}. A notion of disjointness in pre-Riesz spaces has been developed in \cite{KalGaa2006}, which yields a corresponding notion of disjointness preserving operator. Below we use disjointness to define local operators on pre-Riesz spaces in the spirit of Arendt \cite{Are1986C}. 

Our main result is that the generator of a disjointness preserving $\mathrm{C}_0$-semigroup on a suitable ordered Banach space is local, as in the Banach lattice case.  It turns out that the choice of norms is the main issue of the analysis. We consider semimonotone norms for which the cone of positive elements is closed. On Banach spaces, those norms are equivalent to regular norms, which are a natural generalization of lattice norms. Properties of the norms are discussed in Section \ref{npovs}.

Section \ref{lo} contains a discussion of local operators in different settings and some of their basic properties. In Section \ref{dpcs} we present our main result.
Moreover, we establish two results on local operators generating disjointness preserving $\mathrm{C}_0$-semigroups. The first one considers a bounded generator and uses Taylor series, the second one uses resolvent operators and Yosida approximations.

\section{Preliminaries}\label{sec.preli}

In this section we list the terminologies on ordered
vector spaces and vector lattices. Let $X$ be a real vector 
space and let $K$ be a \emph{cone} in $X$, that is $K$ is a  \emph{wedge}
($x,y\in K$, $\lambda,\mu\ge 0$ imply $\lambda x+\mu y\in K$)
and $K\cap(-K)=\{0\}$. In $X$ a partial order is introduced by 
defining $x\le y$ if and only if $y-x\in K$; we write $(X, K)$ for a partially ordered vector space. We call  $(X, K)$ \emph{Archimedean} if for 
every $x, y\in X$ with $nx\le y$ for all $n \in \mathbb{N}$ one has that $x\le 0$. We say that $(X,K)$ is \textit{directed} or $K$ is \textit{generating} in $X$ if $X=K-K$. Denote for a set $M\subseteq X$ the set of all upper bounds of $M$ by
 \[M^\mathrm{u}=\{x\in X\colon x\ge m \mbox{ for all } m\in M\}\]
and the set of all lower bounds by $M^\mathrm{l}$.
 For standard notations in the case that $(X,K)$ is a vector lattice see \cite{AliBur1985}.

By a subspace of a partially ordered vector space or a vector lattice we mean
an arbitrary linear subspace with the inherited order. We do not require it to be 
a lattice or a sublattice. Recall that a subspace $D$ of a partially ordered vector space $Y$ is \emph{majorizing} if for every $y\in Y$ there is 
  $d\in D$ such that $d\ge y$. We call a subspace $D$ of a partially ordered vector space 
$Y$ \emph{order dense} in $Y$ if every $y\in Y$ is the greatest lower bounded of 
 the set $\{d\in D\colon y\le d\}$, i.e. 
 \[y=\inf \{d\in D\colon y\le d\}.\]
Every order dense subspace of a partially ordered vector space is majorizing.

We continue by a notion which is closely related to the order dense embedding of 
a partially ordered vector space into a vector lattice.  A partially order vector space $X$ is called a \emph{pre-Riesz space}
if for every $x,y,z\in X$ the inclusion $\{x+y,x+z\}^\mathrm{u}\subseteq\{y,z\}^\mathrm{u}$ implies 
$x\in K$ \cite[Definition 1.1(viii), Theorem 4.15]{Haa1993}. Every pre-Riesz space is directed
and every directed Archimedean partially ordered vector space is a pre-Riesz space \cite{Haa1993}.
 Clearly, each vector lattice is a pre-Riesz space. Many examples of pre-Riesz spaces that are not vector lattices are given in \cite{KalLemGaa2014, KalGaa2008a}.

 Recall that a linear map $i\colon X\rightarrow Y$, where $X$ and $Y$ are partially ordered vector spaces, is called \emph{bipositive} if for every $x\in X$ one has $0\le x$
 if and only if $0\le i(x)$. An embedding map is required to be linear and bipositive, 
 which implies injectivity. For sets $L\subseteq X$ and $M\subseteq Y$ we denote 
 $i[L]\colon=\{i(x);\, x\in L\}$ and  $i^{-1}[M]\colon=\{x\in X;\, i(x)\in M\}$. 

\begin{theorem}\cite[Corollaries 4.9-11 and Theorems 3.5, 3.7, 4.13]{Haa1993}\label{embd-preR}
Let $X$ be a partially ordered vector space. The following statements are equivalent.
\begin{itemize}
\item[(i)] $X$ is a pre-Riesz space.
\item[(ii)] There exist a vector lattice $Y$ and a bipositive linear map $i\colon X\rightarrow Y$ such that $i[X]$ is order dense in $Y$.
\item[(iii)] There exist a vector lattice $Y$ and a bipositive linear map $i\colon X\rightarrow Y$ such that $i[X]$ is order dense in $Y$ and $i[X]$ generates $Y$ as a vector lattice, i.e. for every $y\in Y$ there are $a_1, \cdots a_m, b_1,\cdots, b_n\in i[X]$ such that 
\[y=\bigvee_{i=1}^{m}a_i-\bigvee_{i=1}^{n}b_i.\]
\end{itemize}
\end{theorem}

A pair $(Y, i)$ as in (ii) is called a \emph{vector lattice cover} of $X$, a pair 
$(Y,i)$ as in (iii) is called a \emph{Riesz completion} of $X$. Since all spaces $Y$ as 
in (iii) are isomorphic as vector lattices \cite[Remark 3.2]{Haa1993}, we will speak of  \emph{the} Riesz completion of 
$X$ and denote it by $(X^\rho, i)$. If $X$ is a pre-Riesz space and $(Y,i)$ is a vector lattice cover of $X$, then $X^\rho$ is the Riesz subspace of $Y$ generated by $i[X]$. Notice that the map $i\colon X\to Y$ in Theorem \ref{embd-preR}(ii) is even a \emph{Riesz* homomorphism}, which means that for every nonempty finite subset $M\subseteq X$ we have that
\[i[M^\mathrm{ul}]\subseteq i[M]^\mathrm{ul},\]
see \cite{Haa1993}.

Disjointness in a partially ordered vector space $(X,K)$ is introduced in \cite{KalGaa2006}. Two elements $x,y\in X$ are called \emph{disjoint}, in symbols $x\perp y$, if 
\[\{x+y,-x-y\}^\mathrm{u}=\{x-y,-x+y\}^\mathrm{u}.\]
If $X$ is a vector lattice, then this notion of disjointness coincides with the usual one
 \cite[Theorem 1.4(4)]{AliBur1985}. 
\begin{proposition}\cite[Proposition 2.1]{KalGaa2006}\label{homo-disj}
Let $X$ be a pre-Riesz space and $(Y,i)$ a vector lattice cover of $X$. Then for every $x,y\in X$ we have
\[
x\perp y\Leftrightarrow i(x)\perp i(y).
\]
\end{proposition}

 The \emph{disjoint complement} of a subset 
$M\subseteq X$ is the set $M^\mathrm{d}=\{y\in X;\, y\perp x \mbox{ for all } x\in M\}$. For a set $S\subseteq X$ Proposition \ref{homo-disj} implies
\begin{equation}
S^\mathrm{d}=i^{-1}[i[S]^\mathrm{d}].
\end{equation}
Thus, disjoint complements in pre-Riesz spaces have properties as in 
 the vector lattice setting, see \cite[Theorem 5.10]{KalGaa2008a}. In particular, a disjoint
 complement is a linear subspace in $X$. More than that, it is a band. Recall that a subspace $B$ of $X$ is a  \emph{band}
 in $X$ if $B=B^\mathrm{dd}$, and note that this notion coincides with the usual one provided $X$ is an Archimedean vector lattice. 


\hide{ 
\begin{theorem}\cite{KalGaa2008b}\label{perv-ext}
Let $X$ be a pre-Riesz space and $(Y,i)$ be its vector lattice 
cover. If $B$ is a band in $X$, then there exists a band $D$
in $Y$ such that $B=i^{-1}[D]$. If, moreover, $X$ is pervasive
and $D$ is band in $Y$, then $i^{-1}[D]$ is a band in $X$.
\end{theorem} 
 } 

\section{Normed partially ordered vector spaces}\label{npovs}

Let $(X,K)$ be a partially ordered vector space. A norm $\left\|\cdot\right\|$ on $X$ is called \emph{monotone} if for every $x,y\in X$ with $0\le x\le y$ one has $\|x\|\le \|y\|$, and \emph{semimonotone} if there exists a constant $M\in\mathbb{R}$ such that for every $x,y\in X$ with $0\le x\le y$ one has $\|x\|\le M\|y\|$. If $\left\|\cdot\right\|$ is a semimonotone norm on $X$, then there exists an equivalent monotone norm on $X$, see \cite[Theorems IV.2.1 and IV.2.4]{Vul1977}. If $X$ is directed, then we say that a seminorm $\left\|\cdot\right\|$ on $X$ is \emph{regular} if $\|x\|=\inf\{\|y\|\colon y\in X, -y\le x\le y\}$ for every $x\in X$. If $X$ is a vector lattice, then a seminorm $\left\|\cdot\right\|$ on $X$ is called a \emph{Riesz norm} if it is monotone and $\| |x| \|=\|x\|$ for every $x\in X$. On vector lattices, the notions of regular seminorms and Riesz seminorms coincide. 

\begin{lemma}\label{regular-extension}
Let $(X,K)$ be a directed partially ordered vector space and let $\left\|\cdot\right\|$ be a semimonotone norm on $X$. Let $Y$ be a directed partially ordered vector space and $i\colon X\to Y$ a bipositive linear map, such that $i[X]$ is majorizing in $Y$. For $y\in Y$ let
\begin{equation}\label{norm-ext}
\rho(y):=\inf\left\{\|x\|;\  x\in X, \,-i(x)\le y\le i(x)\right\}.
\end{equation}
The following statements hold.
\begin{itemize}
\item[(i)] $\rho$ is a regular seminorm on $Y$. 
\item[(ii)] If $\left\|\cdot\right\|$ is regular, then $\rho$ extends $\left\|\cdot\right\|$ in the sense that $\rho\circ i=\left\|\cdot\right\|$. 
\item[(iii)] If $Y=X$, $K$ is closed and $(X,\left\|\cdot\right\|)$ is complete, then $\rho\circ i$ is a regular norm on $X$ that is equivalent to $\left\|\cdot\right\|$.
\end{itemize}
\end{lemma} 
\begin{proof}
It is contained in \cite[Proposition 6.2]{Gaa2004} that $\rho$ is a regular seminorm and that $\rho\circ i=\left\|\cdot\right\|$ if $\left\|\cdot\right\|$ is regular. By \cite[Theorems IV.2.1 and IV.2.4]{Vul1977}, $\left\|\cdot\right\|$ is equivalent to a monotone norm and then, if $(X,\left\|\cdot\right\|)$ is complete and $K$ is closed, \cite[Corollary 6.4(ii)]{Gaa2004} says that $\left\|\cdot\right\|$ is equivalent to the regular norm $\rho\circ i$.
\end{proof}

The last statement of Lemma \ref{regular-extension} is in fact on an ordered Banach space. By an \emph{ordered Banach space}  $(X,K,\left\|\cdot\right\|)$ we  mean a Banach space $(X,\left\|\cdot\right\|)$ with a closed generating cone $K$. Since $K$ is closed, the space $(X,K)$ is Archimedean, and since $K$ is generating, $(X,K)$ is directed. Consequently, $(X,K)$ is a pre-Riesz space and can be embedded into its Riesz completion $X^\rho$, see Theorem \ref{embd-preR}.

The next lemma shows how disjointness can be detected with the aid of the extension given by \eqref{norm-ext}.

\begin{lemma}
\label{lem:char_perp}
Let $(X,K,\left\|\cdot\right\|)$ be a normed partially ordered vector space such that $K$ is closed and generating. Let $Y$ be a vector lattice and
$i\colon X\to Y$ a bipositive linear map, such that $i[X]$ is majorizing in $Y$.
For $y\in Y$ let $\rho(y)$ be defined by \eqref{norm-ext}.
If $u,v\in X$ are such that $\rho(|i(u)|\wedge|i(v)|)=0$, then $u\perp v$.
\end{lemma}  
\begin{proof}
Let $u,v\in X$ be such that $\rho(|i(u)|\wedge|i(v)|)=0$.
As \[|i(u)|\wedge |i(v)|=\textstyle\frac{1}{2}\left|\, \left|i(u)+i(v)\right|-\left|i(u)-i(v)\right| \,\right|\] one obtains that
$\rho\left(|i(u)+i(v)|-|i(u)-i(v)|\right)=0$.
Hence for every $n\in \N$ there exists an $x_n\in X$ such that
\begin{equation}\label{equ:i}
-i(x_n)\le |i(u)+i(v)|-|i(u)-i(v)|\le i(x_n)
\end{equation}
and $\lim_{n\to \infty}\|x_n\|= 0$. The first inequality in \eqref{equ:i} yields
\[\pm(i(u)-i(v))-i(x_n)\le |i(u)+i(v)|,\]
hence for $x\ge \pm(u+v)$ it follows that $x\ge \pm(u-v)-x_n$ for every $n\in\N$.
Since $K$ is closed and $x_n\to 0$, one obtains $x\ge \pm(u-v)$.
We conclude 
\begin{equation}\label{equ:ii}
\{u+v,-u-v\}^\mathrm{u}\subseteq\{u-v,-u+v\}^\mathrm{u}.
\end{equation}
From the second inequality in \eqref{equ:i} it follows that
\[\pm(i(u)+i(v))-i(x_n)\le |i(u)-i(v)|,\]
and an analogous argumentation yields equality in \eqref{equ:ii},
which implies that $u$ and $v$ are disjoint. 
\end{proof}

It turns out that bands are closed for regular norms.

\begin{lemma}\label{clo-band}
If
\begin{itemize}
\item[(i)] $(X,K)$ is a pre-Riesz space with a regular norm $\left\|\cdot\right\|$ such that $K$ is closed, or
\item[(ii)] $(X,K,\left\|\cdot\right\|)$ is an ordered Banach space with a semimonotone norm, 
\end{itemize}
then every band in $X$ is closed.
\end{lemma}
\begin{proof}
According to Lemma \ref{regular-extension}, the conditions in (ii) yield that the norm $\left\|\cdot\right\|$ is equivalent to a regular norm, so it suffices to give a proof under condition (i). 

Let $B$ be a band in $X$, let $(x_n)$ be a sequence in $B$ and let $x\in X$ be such that $\|x_n-x\|\to 0$.  Let $(X^\rho,i)$ be the Riesz completion of $X$, as in Theorem \ref{embd-preR}. Let $\rho$ be the Riesz seminorm on $Y=X^\rho$ defined as in \eqref{norm-ext}.

Let $y\in B^\mathrm{d}$. Then for every $n\in \mathbb{N}$ one has $x_n\perp y$, so that by Proposition \ref{homo-disj} one has $i(x_n)\perp i(y)$, which implies $|i(x_n)|\wedge |i(y)|=0$. Since $\rho(i(x_n))-i(x))=\|x_n-x\|\to 0$, the continuity of the lattice operations with respect to $\rho$ yields that $\rho(|i(x)|\wedge |i(y)|)=0$. Hence, by Lemma \ref{lem:char_perp}, we obtain $x\perp y$. It follows that $x\in B^\mathrm{dd}=B$. Hence $B$ is closed.
\end{proof}

\section{Local operators}\label{lo}

In this section we introduce a notion of local operators on pre-Riesz spaces. Local operators turn out to be a special class of disjointness preserving operators, whose definition we recall first.

\begin{definition}
Let $X$ and $Y$ be partially ordered vector spaces and let $T\colon X\rightarrow 
Y$ be a linear operator. $T$ is called \emph{disjointness preserving} if for every 
$x, y\in X$ from $x\perp y$ in $X$ it follows that $Tx\perp Ty$ in $Y$.
\end{definition}


The class of  local operators plays a role in the theory of differential equations, where the notion `local' is used ambiguously (see Remark \ref{ambig_local} below).
Also band preserving operators 
are a well-established class of operators in the theory of vector lattices, see e.g.\ \cite{AliBur1985,Mey1991}. 
In \cite[(5.4)]{Are1986C}, local operators are defined on Banach lattices.   

\begin{definition}\label{def:local}
Let $X$  be a partially ordered vector space and let $T\colon X\supseteq \mathcal{D}(T)\to X$ be a linear operator.

\begin{itemize}
\item[(i)]  $T$ is called \index{band preserving operator}
\emph{band preserving} if for every band $B$ in $X$ one has $T(B\cap \mathcal{D}(T))\subseteq B$. 
\item[(ii)] $T$ is called \emph{local} if for every $x\in \mathcal{D}(T)$, $y\in X$ with $x\perp y$ it follows that $Tx\perp y$. \index{local operator}
\end{itemize}	
\end{definition}

Observe that
if $T$ is local and $x,y\in \mathcal{D}(T)$ are such that $x\perp y$, then $Tx\perp y$ and, moreover, $Tx\perp Ty$, i.e.\
$T$ is disjointness preserving. With above definition, we observe that local operators and band preserving operators coincide in pre-Riesz space.

\begin{proposition} 
Suppose that $X$ is a pre-Riesz space and let $T\colon X\supseteq\mathcal{D}(T)\rightarrow X$ be a linear operator.
$T$ is local if and only if
$T$ is band preserving.
\end{proposition}

\begin{proof}
Let $T$ be local. For every $x\in \mathcal{D}(T)$
one obtains $Tx\in \{x\}^{\operatorname{dd}}$. Indeed, for every $y\in \{x\}^{\operatorname{d}}$ it follows that $Tx\perp y$,
therefore $Tx\in \{x\}^{\operatorname{dd}}$.
Let $B$ be a band in $X$ and $x\in B\cap\mathcal{D}(T)$. Then $\{x\}^{\operatorname{dd}}\subseteq B^{\operatorname{dd}}=B$, hence $Tx\in B$. We conclude that $T$ is band preserving. 

Now let $T$ be band preserving, $x\in \mathcal{D}(T)$ and $y\in X$ such that $x\perp y$. Then $\{y\}^{\operatorname{d}}$ is a band 
in $X$, and $x\in \{y\}^{\operatorname{d}}\cap \mathcal{D}(T)$, hence $Tx\in \{y\}^{\operatorname{d}}$, which yields $Tx\perp y$.
Consequently $T$ is local.
\end{proof}
     
Typical examples of local operators are differential operators and multiplication operators. 
We discuss some of these settings in the subsequent remarks.
Let $\mathcal{L}(X)$ denote the bounded linear operators on $X$.

\begin{remark}\label{rem:mult_opt}
If $X$ is a Banach lattice, every bounded local operator on $X$
is contained in the \emph{center} \[\mathrm{Z}(X):=\{T\in \mathcal{L}(X); \ \exists \alpha>0 \colon -\alpha I\le T\le \alpha I \}.\] 
More precisely, the following assertions are equivalent for a 
bounded linear operator $T\colon X\to X$ (see e.g.\ \cite[Section 9]{NagSch1986C}):
\begin{itemize}
\item[(i)] $T$ is local,
\item[(ii)] $-\|T\|I\le T \le \|T\|I$,
\item[(iii)] for every ideal $J$ in X one has $T[J]\subseteq J$.
\end{itemize}
In \cite[Section 9]{NagSch1986C} the following two examples are given to illustrate that local operators are closely related to multiplication operators.
\begin{itemize}
\item[(a)] Let $(\Omega,\mu)$ be a $\sigma$-finite measure space and 
$X=\mathrm{L}^p(\Omega,\mu)$ ($1\le p\le \infty$), then
$\mathrm{Z}(X)$ is isomorphic to $\mathrm{L}^\infty(\Omega,\mu)$ via the
identification of $y\in \mathrm{L}^\infty(\Omega,\mu)$ with the operator from $X$ to $X$ given by $x\mapsto yx$.

\item[(b)] Let $\Omega$ be a locally compact Hausdorff space and $X$ the space $\mathrm{C}_0(\Omega)$ of all continuous functions $x$ on $\Omega$ vanishing at infinity (i.e.\ for every $\varepsilon >0$ there is a compact set $C\subseteq \Omega$ such that for every $t\in\Omega\setminus C$ one has $|x(t)|<\varepsilon$), endowed with the supremum norm. $\mathrm{Z}(X)$ is isomorphic to the space $\mathrm{C}^\mathrm{b}(\Omega)$ of bounded continuous functions via the
identification of $y\in \mathrm{C}^\mathrm{b}(\Omega)$ with the operator
from $X$ to $X$ given by $x\mapsto yx$.
\end{itemize}
\end{remark}

\begin{remark}\label{propertyL}
We continue the discussion of case (a) above for unbounded $T$. 
Let $X=\mathrm{L}^p(\Omega,\mu)$ ($1\le p\le \infty$) and $T\colon  \mathrm{L}^p(\Omega,\mu)\supseteq \mathcal{D}(T)\to \mathrm{L}^p(\Omega,\mu)$.
A notion of `locality' in this setting is defined in \cite[I.4.13(8)]{EngNag2000}. To avoid confusion, we denote this property by (L):
 
\begin{itemize}
\item[(L)] For every measurable set $S\subseteq \Omega$ and for every $x\in \mathcal{D}(T)$ with $x=0$ almost everywhere on $S$ it follows that $Tx=0$ almost everywhere on $S$. 
\end{itemize}
$T$ is local if and only if (L) is satisfied. Indeed, suppose that (L) is satisfied and let 
$x\in \mathcal{D}(T)$, $y\in X$ and
$x\perp y$. Then $S:=\{t\in \Omega; y(t)\neq 0\}$ is measurable
and $x=0$ almost everywhere on $S$, hence $Tx=0$ almost everywhere on $S$, which implies $Tx\perp y$. Therefore 
$T$ is local. Now suppose that $T$ is local and let $S\subseteq \Omega$ be measurable and $x\in \mathcal{D}(T)$ be such that $x=0$ almost everywhere on $S$.
Then $x\perp \chi_S$, hence $Tx \perp \chi_S$, which implies 
$Tx=0$ on $S$. Consequently (L) is satisfied.
In the present setting, a local operator need not be a multiplication operator, 
consider e.g.\ the operator $T\colon \mathrm{L}^p[0,1]\supseteq \mathrm{C}^1[0,1] \to \mathrm{L}^p[0,1]$, $x\mapsto x'$. 
\end{remark}

\begin{remark}\label{ambig_local}
We continue the discussion of (b) in Remark \ref{rem:mult_opt},
where we now consider the special case where $\Omega$ is a compact Hausdorff space and $X=\mathrm{C}(\Omega)$. For every bounded local operator $T\colon \mathrm{C}(\Omega)\to \mathrm{C}(\Omega)$ there is $y\in \mathrm{C}(\Omega)$ such that $T\colon x\mapsto yx$. (This can be deduced from the fact that $\mathrm{C}(\Omega)$ is an Archimedean f-algebra with unit and that $T$ is band preserving and order bounded, see  
 \cite[Theorem 8.27]{{AliBur1985}}.) We discuss several notions of locality for unbounded operators $T$.

(I) In \cite[Theorem 3.7]{Are1986B} a linear operator $T\colon \mathrm{C}(\Omega) \supseteq\mathcal{D}(T)\to \mathrm{C}(\Omega)$ is called \textit{local} if  
the following property is satisfied:
\begin{equation}\label{Equ:i}
\forall \, x\in \mathcal{D}(T),\ x\ge 0, \ \omega\in\Omega \ \mbox{ with } \ x(\omega)=0 \ \Rightarrow \ (Tx)(\omega)=0. 
\end{equation}
If \eqref{Equ:i} is satisfied for an operator $T$, then $T$ is local in the sense of Definition \ref{def:local}. The converse implication is not true, in general.
Indeed, consider $\mathcal{D}(T):=\{x\in \mathrm{C}^2[0,1]; x'(0)=x'(1)=0 \}$ and \[T\colon \mathrm{C}[0,1]\supseteq \mathcal{D}(T)\to \mathrm{C}[0,1], \ x\mapsto x''\]
(cf. the one-dimensional diffusion semigroup in \cite[2.7]{NagSch1986}). On one hand, $T$ does not satisfy \eqref{Equ:i}, take e.g.\ $x(t)=\left(t-\frac{1}{2}\right)^2$.
On the other hand, $T$ is local. Indeed, let $x\in \mathcal{D}(T)$, $y\in \mathrm{C}[0,1]$ with $x\perp y$, and $N:=\{t\in[0,1]; x(t)=0\}$. For $t\in \operatorname{int}N$ one has $(Tx)(t)=0$, whereas for $t\in \partial N$ and every $n\in \N$ there is $t_n\in\left[t-\frac{1}{n}, t+\frac{1}{n}\right]\cap [0,1]$ such that $x(t_n)\neq 0$, i.e.\ $y(t_n)=0$, which implies $y(t)=0$. Hence $Tx\perp y$, consequently $T$ is local. 

(II) An operator $T\colon \mathrm{C}(\Omega) \supseteq\mathcal{D}(T)\to \mathrm{C}(\Omega)$ satisfies \eqref{Equ:i} if and only if for every open set $O\subseteq \Omega$ one has $T(I_O\cap \mathcal{D}(T))\subseteq I_O$, where \[I_O=\{x\in \mathrm{C}(\Omega); \forall \omega\in\Omega\setminus O\colon x(\omega)=0\},\] i.e.\ $T$ preserves for every open set $O\subseteq \Omega$ the largest ideal having $O$ as its carrier.

(III) We relate \eqref{Equ:i} to a notion of locality given in \cite[p.147, second Remark]{Are1986B}.
Let $(X,K)$ be a partially ordered vector space 
with a norm $\left\|\cdot\right\|$.
For a linear operator $T\colon X\supseteq \mathcal{D}(T)\to X$ the following property is considered:
\begin{equation}\label{Equ:ii}
\forall \, x\in \mathcal{D}(T)\cap K, \ f\in K' \ \mbox{ with } \ f(x)=0 \ \Rightarrow \ f(Tx)= 0. 
\end{equation}
Note that \eqref{Equ:ii} is satisfied if and only if $T$ and $-T$ are \textit{positive-off-diagonal} (for a definition of this notion see  \cite[Definition 7.18]{CleHeiAngDuiPag1987}; cf.\ also the \emph{positive minimum principle} in \cite[Definition 1.6]{Are1986C}).

Before we link \eqref{Equ:i} and \eqref{Equ:ii}, we reformulate 
\eqref{Equ:ii} for the case that $X$ is an Archimedean partially ordered vector space with order unit $u$. The norm induced by $u$, which
will be denoted by $\left\|\cdot\right\|_u$, is defined by
\[\left\|x\right\|_u=\inf\{\alpha\in [0,\infty)\colon\, -\alpha u\le x\le \alpha u\}, x\in X.\]
Let \[\Sigma\colon=\{f\colon X\to \mathbb{R};\ f \mbox{ positive }\mbox{ linear}, \mbox{ and } f(u)=1\},\]
and denote $\Lambda$ by the set of extreme points of $\Sigma$.
For $x\in \mathcal{D}(T)\cap K$
the set $N=\{f\in \Sigma;\, f(x)=0\}$ is the weak* closure of the convex hull of $N\cap \Lambda$ (see also \cite[(3.6)]{Kal2006}).
Therefore
\eqref{Equ:ii} holds if and only if 
\begin{equation}\label{Equ:iii}
\forall \, x\in \mathcal{D}(T)\cap K, \ f\in \Lambda \ \mbox{ with } \ f(x)=0 \ \Rightarrow \ f(Tx)= 0. 
\end{equation}
For $X=\mathrm{C}(\Omega)$ as above and $u$ the constant-one function, $\Lambda$ is homeomorphic to $\Omega$, hence 
the conditions \eqref{Equ:i} and \eqref{Equ:iii} are equivalent.
This implies the equivalence of \eqref{Equ:i} and \eqref{Equ:ii}.
\end{remark}


Now let us consider some basic results on local operators on pre-Riesz spaces. The following lemma will be needed in the proof of Theorem \ref{bddlocalgenerator}.

\begin{lemma}\label{sum-local}
Let $X$ be a pre-Riesz space.
\begin{itemize}
\item[(i)] If $S\colon X\supseteq \mathcal{D}(S)\to X$ and $T\colon X\supseteq \mathcal{D}(T)\to X$ are local operators and $\alpha,\beta\in\mathbb{R}$, then $\alpha S+\beta T\colon X\supseteq \mathcal{D}(S)\cap\mathcal{D}(T)\to X$ is a local operator.
\item[(ii)] If $S\colon X\supseteq \mathcal{D}(S)\to X$ and $T\colon X \supseteq \mathcal{D}(T)\to \mathcal{D}(S)\subseteq X$ are local operators, then $ST\colon X \supseteq \mathcal{D}(T)\to X$ is a local operator.
\end{itemize}
\end{lemma}
\begin{proof}
(i) Let $x\in \mathcal{D}(S)\cap\mathcal{D}(T)$ and $y\in X$ be such that $x\perp y$. Then $Sx\perp y$ and $Tx\perp y$, so that $Sx,Tx\in \{y\}^\mathrm{d}$. Since $\{y\}^\mathrm{d}$ is a linear subspace, it follows that $\alpha Sx+\beta Tx\perp y$. Hence $\alpha S+\beta T$ is local.

(ii) Let $x\in \mathcal{D}(T),y\in X$ be such that $x\perp y$. Then $Tx\perp y$ as $T$ is local, so that $STx\perp y$ as $S$ is local and $Tx\in \mathcal{D}(S)$. Hence $ST$ is local.
\end{proof}

Next we consider a result on the inverse of a local operator, which we need in the proof of Corollary \ref{posresolvent} below. It turns out that the inverse of a local operator $T$ is local if both $T$ and $T^{-1}$ are positive. We first need the next lemma. 

\begin{lemma}\label{biposRieszstar}
If $X$ and $Y$ are pre-Riesz spaces and $T\colon X\to Y$ is a bipositive Riesz* homomorphism, then for every $x,y\in X$ we have
\[x\perp y \ \Longleftrightarrow\ Tx\perp Ty.\]
\end{lemma}
\begin{proof}
$\Leftarrow$) If $Tx\perp Ty$ in $Y$, then it follows directly from the definition of disjointness  that $Tx\perp Ty$ in $T[X]$ (see also \cite[Proposition 2.1(i)]{KalGaa2006}), so that bipositivity of $T$ yields that $x\perp y$.

$\Rightarrow$) If $(X^{\rho}, i_X)$ and $(Y^{\rho},i_Y)$ are the Riesz completions of $X$ and $Y$, respectively, then due to \cite[Theorem 5.6]{Haa1993} 
there is a linear Riesz homomorphism $T_{\rho}\colon X^{\rho}\to Y^{\rho}$ that extends $T$ in the sense that $T_{\rho}\circ i_X= i_Y\circ T$.
In particular, $T_{\rho}$ is positive and disjointness preserving. Recall that $x\perp y$ in $X$ if and only if $i_X(x)\perp i_X(y)$ in $X^\rho$, and similarly for elements in $Y$. For $x,y\in X$ with $x\perp y$ we thus have $i_X(x)\perp i_X(y)$ and therefore $T_\rho (i_X(x))\perp T_\rho (i_X(y))$, so that $i_Y (T(x))\perp i_Y(T(y))$, which yields that $T(x)\perp T(y)$.
\end{proof}

\begin{proposition}\label{inv-loc}
Let $(X,K)$ be a pre-Riesz space, let $T\colon X\supseteq \mathcal{D}(T)\rightarrow X$ be a bijective linear operator such that the inclusion map $i\colon \mathcal{D}(T)\to X$ is a Riesz* homomorphism. If both $T$ and $T^{-1}$ are positive and $T$ is local, then $T^{-1}$ is also local.
\end{proposition}
\begin{proof}
As a first step, for $x\in X, y\in \mathcal{D}(T)$ with $x\perp y$ we show that $T^{-1}x\perp y$ in $\mathcal{D}(T)$, which comes down to $\{T^{-1}x+y,-T^{-1}x-y\}^\mathrm{u}\cap\mathcal{D}(T)=\{T^{-1}x-y,-T^{-1}x+y\}^\mathrm{u}\cap\mathcal{D}(T)$. Let $z\in \{T^{-1}x+y,-T^{-1}x-y\}^\mathrm{u}\cap\mathcal{D}(T)$. Then, as $T$ is positive, $Tz\ge x+Ty,-x-Ty$. Since $T$ is local and therefore $Ty\perp x$, we have $Tz\ge x-Ty,-x+Ty$, so $z\ge T^{-1}x-y, -T^{-1}x+y$, as $T^{-1}$ is positive. Therefore, $\{T^{-1}x+y,-T^{-1}x-y\}^\mathrm{u}\cap\mathcal{D}(T)\subseteq \{T^{-1}x-y,-T^{-1}x+y\}^\mathrm{u}\cap\mathcal{D}(T)$. A similar proof yields the converse inclusion, so that $\{T^{-1}x+y,-T^{-1}x-y\}^\mathrm{u}\cap\mathcal{D}(T)= \{T^{-1}x-y,-T^{-1}x+y\}^\mathrm{u}\cap\mathcal{D}(T)$. This shows that
$T^{-1}x\perp y$ in $\mathcal{D}(T)$.

Since the inclusion map $i\colon \mathcal{D}(T)\to X$ is a Riesz* homomorphism, Lemma \ref{biposRieszstar} yields that $i(T^{-1}x)\perp i(y)$, hence $T^{-1}x\perp y$ in $X$. Thus, $T^{-1}$ is local.
\end{proof}


We conclude this section by the following simple observation.

\begin{lemma}\label{genr-disj}
Let $(X,K)$ be a pre-Riesz space and $T\colon X\rightarrow X$ a bipositive linear bijection. Then $T$ is disjointness preserving.
\end{lemma}
\begin{proof}
Let $x,y\in X$ be such that $x\perp y$. Then $\{x+y,-x-y\}^\mathrm{u}=\{x-y,-x+y\}^\mathrm{u}$. Hence for every $u\in X$ we have 
\[u\ge x+y,-x-y \ \Longleftrightarrow\ u\ge x-y, -x+y.\]
As $T$ is bipositive, the latter is equivalent to
\[Tu\ge Tx+Ty,-Tx-Ty \ \Longleftrightarrow\ Tu\ge Tx-Ty, -Tx+Ty.\]
Since $T$ is surjective, this comes down to $\{Tx+Ty,-Tx-Ty\}^\mathrm{u}=\{Tx-Ty,-Tx+Ty\}^\mathrm{u}$. Hence $Tx\perp Ty$. 
\end{proof}




\section{Disjointness preserving $\mathrm{C}_0$-semigroups}\label{dpcs}
Let  $(X,K,\left\|\cdot\right\|)$ be an ordered Banach space. A $\mathrm{C}_0$-semigroup $T\colon [0,\infty)\to \mathcal{L}(X)$ 
is called \emph{disjointness preserving} if for every $t\in[0,\infty)$ the operator $T(t)$ is disjointness preserving. \index{disjointness preserving $\mathrm{C}_0$-semigroup}
Disjointness preserving $\mathrm{C}_0$-semigroups on Banach lattices are discussed e.g.\ in \cite[Section 5]{Are1986C}, where, in particular, it is shown that 
the generator of a disjointness preserving $\mathrm{C}_0$-semigroup is local.
We prove the analogous result for disjointness preserving $\mathrm{C}_0$-semigroups on ordered Banach spaces with a semimonotone norm.

\begin{theorem}
\label{the:op_sg_local}
Let $(X,K,\left\|\cdot\right\|)$ be an ordered Banach space with a semimonotone norm and let 
$T\colon [0,\infty)\to \mathcal{L}(X)$ be a disjointness preserving $\mathrm{C}_0$-semigroup with generator $A$. Then $A$ is local.
\end{theorem}

\begin{proof}
Let $(X^\rho,i)$ be the Riesz completion of $X$. 
Because of Lemma \ref{regular-extension}, the semimonotone norm $\left\|\cdot\right\|$ is equivalent to the regular norm $\rho\circ i$, where the regular norm $\rho$ on $Y=X^\rho$ is given by \eqref{norm-ext}. As $Y$ is a vector lattice, $\rho$ is a Riesz norm and \eqref{norm-ext} comes down to
\[\rho(y)=\inf\left\{\|x\|;\, x\in X, \,|y|\le i(x)\right\},\ y\in Y.\]
Let $x\in \mathcal{D}(A)$ and $y\in X$ be such that $x\perp y$. Then $i(x)\perp i(y)$ in $Y$. 
Now the line of reasoning is as in the proof of \cite[Proposition 5.4]{Are1986C}. For every $t>0$ one has  
\begin{eqnarray*}
\left|\textstyle\frac{1}{t} i(T(t)x-x)\right|\wedge|i(y)|&\le&
\textstyle\frac{1}{t} \left|i(T(t)x)\right|\wedge|i(y)|+
\textstyle\frac{1}{t} \left|i(x)\right|\wedge|i(y)|\\
&=& \textstyle\frac{1}{t} \left|i(T(t)x)\right|\wedge|i(T(t)y-y)-i(T(t)y)|\\
&\le&\textstyle\frac{1}{t} \left|i(T(t)x)\right|\wedge|i(T(t)y-y)|\\
&\le& |i(T(t)y-y)|.
\end{eqnarray*}
Here we used that $Y$ is distributive and $T$ is disjointness preserving.
We conclude \[\rho\left(\left|\textstyle\frac{1}{t} i(T(t)x-x)\right|\wedge|i(y)|\right)\le \rho(|i(T(t)y-y)|).\]
For $t\downarrow 0$ one has $T(t)y-y\to 0$ in $X$, hence
$\rho(|i(T(t)y-y)|)\to 0$, which implies 
\[\rho\left(\left|\textstyle\frac{1}{t} i(T(t)x-x)\right|\wedge|i(y)|\right)\to 0.\] Further, 
\begin{eqnarray*}& &
\left|\rho\left(\left| i(Ax)\right|\wedge|i(y)|\right)-\rho\left(\left|\textstyle\frac{1}{t} i(T(t)x-x)\right|\wedge|i(y)|\right)\right|\\
&\le&
\rho\left(\left|\left| i(Ax)\right|\wedge|i(y)|-\left|\textstyle\frac{1}{t} i(T(t)x-x)\right|\wedge|i(y)|\right|\right)\\
&\le& \rho \left( \left| \left| i(Ax)\right| -\left| i\left( \textstyle\frac{1}{t} (T(t)x-x) \right) \right| \right| \right)\\
&\le& \rho \left( \left| i\left(Ax -  \textstyle\frac{1}{t} (T(t)x-x) \right)  \right| \right)\to 0,
 \end{eqnarray*}
for $t\downarrow 0$, since $\|Ax -  \textstyle\frac{1}{t} (T(t)x-x)\|\to 0$. Therefore $\rho\left(\left| i(Ax)\right|\wedge|i(y)|\right)=0$. Now Lemma \ref{lem:char_perp} implies that $Ax\perp y$, hence $A$ is local. 
\end{proof}

Notice that the converse of the statement of \ref{the:op_sg_local} is not true in general, not even in Banach lattices. We illustrate this by an example from \cite{NagSch1986, EngNag2000}.

\begin{example}\label{ndp}  
Let $A$ be the second derivative operator given in Remark \ref{ambig_local}(I).
We have already shown that $A$ is local. The one-dimensional diffusion semigroup generated by $A$ is given
by 
\[T(t)f(x)=\int_0^1K_t(x,y)f(x)dy,\]
with kernel
\[K_t(x,y)=1+2\sum_{n=1}^{\infty}\exp(-\pi^2n^2t)\cos (\pi nx)\cdot\cos (\pi ny),\]
see \cite[2.7]{NagSch1986} or \cite[2.12]{EngNag2000}. There it is also shown that
$K_t(\cdot,\cdot)$ is a positive, continuous function on $[0,1]^2$.
Obviously, for $t\in (0,\infty)$, $T(t)$  is not disjointness preserving on $\mathrm{C}[0,1]$.
\end{example}

It is nevertheless interesting to investigate a converse of Theorem \ref{the:op_sg_local}. We consider two cases in which a $\mathrm{C}_0$-semigroup with a local generator will be disjointness preserving. The first one considers as extra condition that the generator is a bounded operator and uses the Taylor series for the semigroup. The second case assumes that also the resolvent operators are local and uses the Yosida approximations. It turns out that in both cases the semigroup even consists of local operators.

We begin with the case of a bounded generator.

\begin{theorem}\label{bddlocalgenerator}
Let $(X,K,\left\|\cdot\right\|)$ be an ordered Banach space with a semimonotone norm. If $A\in\mathcal{L}(X)$ is local, then $e^{tA}$ is local for every $t\in \mathbb{R}$.
\end{theorem}
\begin{proof}
Let $x, y\in X$ be such that $x\perp y$. Let $t\in\mathbb{R}$. For every $N\in\mathbb{N}$, Lemma \ref{sum-local} yields that $\sum_{k=0}^N \frac{t^n}{n!} A^nx\perp y$. According to Lemma \ref{clo-band}, $\{y\}^\mathrm{d}$ is closed, so that $e^{tA}x=\sum_{k=0}^\infty \frac{t^n}{n!} A^nx\perp y$. Hence $e^{tA}$ is local.
\end{proof}

We proceed by investigating unbounded local generators. Typical examples of local operators are differential operators and multiplication operators. Recall that Example \ref{ndp} shows an example of a differential operator that generates a $\mathrm{C}_0$-semigroup that is not disjointness preserving. The next example presents a differential operator that generates a $\mathrm{C}_0$-semigroup which is disjointess preserving. It also presents a multiplication operator as generator.

\begin{example}\label{loc-exa}
(i) Translation Semigroup. We consider the Banach space $X:=\mathrm{C}_{\mathrm{ub}}(\mathbb{R})$ of all uniformly continuous, bounded functions on $\mathbb{R}$. For $t\in [0,\infty)$, the left translation operator $T_l(t)\colon X\to X$ is given by
\[T_l(t)x(s) := x(s + t),\ s\in \mathbb{R},\ x\in X.\]
Then $T_l\colon [0,\infty)\to \mathcal{L}(X)$ is a $\mathrm{C}_0$-semigroup on $X$ with generator $A$ given by differentiation,
\[Ax:=x'\]
with domain
$\mathcal{D}(A)=\{x\in X;\,  x \mbox{ differentiable and } x'\in X\}$.
Then $A$ is local (and unbounded) and $T$ is disjointness preserving, but not  local.

(ii) Multiplication Semigroup. Let $X:=\mathrm{C}_0(\Omega)$, 
where $\Omega$ is a locally compact Hausdorff space, as defined in Remark \ref{rem:mult_opt}(b). Let $q\colon \Omega\to\mathbb{R}$ be continuous and bounded above. Define for $t\in [0,\infty)$ the operator $T_q(t)\colon X\to X$ by
\[T_q(t)x=e^{tq(t)}x,\ x\in X.\]
Then $T_q\colon [0,\infty)\to\mathcal{L}(X)$ is a $\mathrm{C}_0$-semigroup with generator $A$ given by  $Ax=qx$, $x\in \mathcal{D}(A)$ and $\mathcal{D}(A)=\{x\in X\colon\, qx\in X\}$. Then $A$ is local and $T_q(t)$ is local for every $t\in [0,\infty)$.
\end{example}

An interesting difference between the generators $A$ in Example (i) and (ii) above is that in (ii) also the inverse of $A$ is local. It turns out that a $\mathrm{C}_0$-semigroup is local whenever both $A$ and $A^{-1}$ are local, which is a special case of the following theorem.

\begin{theorem}\label{localresolvent}
Let $(X,K,\left\|\cdot\right\|)$ be an ordered Banach space 
with semimonotone norm
and let $T\colon [0,\infty)\to\mathcal{L}(X)$ be a $\mathrm{C}_0$-semigroup with generator $A$. If $A\colon X\supseteq \mathcal{D}(A) \to X$ is local and there exists a $\lambda_0\in \rho(A)\cap \mathbb{R}$ such that for every $\lambda\in\rho(A)$ with $\lambda\ge \lambda_0$ we have that $(\lambda I-A)^{-1}\colon X\to \mathcal{D}(A)\subseteq X$ is local, then $T(t)$ is local for every $t\in [0,\infty)$. 
\end{theorem}
\begin{proof}
By Lemma \ref{sum-local}, the Yosida approximation $A_\lambda =A(\lambda I-A)^{-1}$ is local for every $\lambda\in \rho(A)$ with $\lambda\ge \lambda_0$. Let $t\in [0,\infty)$. Since $A_\lambda$ is bounded, due to Theorem \ref{bddlocalgenerator} we obtain that $e^{t A_\lambda}$ is local. For $x\in X$, we have $T(t)x=\lim_{\lambda\to\infty} e^{tA_\lambda}x$. We infer that $T(t)$ is local. Indeed, if $x,y\in X$ are such that $x\perp y$, then $e^{tA_\lambda}x\perp y$. Since the band $\{y\}^\mathrm{d}$ is closed by Lemma \ref{clo-band}, it follows that $T(t)x\perp y$. Thus, $T(t)$ is local. 
\end{proof}

\begin{corollary}\label{posresolvent}
Let $(X,K,\left\|\cdot\right\|)$ be an ordered Banach space 
with semimonotone norm
and let $T\colon [0,\infty)\to\mathcal{L}(X)$ be a $\mathrm{C}_0$-semigroup with generator $A$ such that the inclusion map $i\colon\mathcal{D}(A)\to X$ is a Riesz* homomorphism. If there exists a $\lambda_0\in \rho(A)\cap \mathbb{R}$ such that $\lambda_0I-A\colon\mathcal{D}(A)\to X$ is positive and local and for every $\lambda\in\rho(A)$ with $\lambda\ge \lambda_0$ we have that $(\lambda I-A)^{-1}$ is positive, then $T(t)$ is local for every $t\in [0,\infty)$. 
\end{corollary}
\begin{proof}
Let $\lambda\in\rho(A)$ with $\lambda\ge\lambda_0$. Then $\lambda I-A\colon\mathcal{D}(A)\to X$ is positive and bijective and, by assumption, $(\lambda I-A)^{-1}$ is positive. By Lemma \ref{sum-local}, $\lambda I-A=\lambda_0 I-A+(\lambda -\lambda_0)I$ is local. Proposition \ref{inv-loc} then yields that $(\lambda I-A)^{-1}$ is local. Hence we can apply Theorem \ref{localresolvent} and obtain that $T(t)$ is local for every $t\in [0,\infty)$.
\end{proof}

\begin{corollary}\label{loc-sg}
Let $(X,K,\left\|\cdot\right\|)$ be an ordered Banach space 
with semimonotone norm
and let $T\colon [0,\infty)\to\mathcal{L}(X)$ be a $\mathrm{C}_0$-semigroup with generator $A$ such that the inclusion map $i\colon\mathcal{D}(A)\to X$ is a Riesz* homomorphism. If $A$ is positive and local and there exists a $\lambda_0\in \rho(A)\cap \mathbb{R}$ such that $A\le \lambda_0 I$, then $T(t)$ is local for every $t\in [0,\infty)$. 
\end{corollary}
\begin{proof}
Assume that $A$ is positive. Then $T(t)$ is positive for every $t\in [0,\infty)$, so there exists $\lambda_1\in\mathbb{R}$ such that $(\lambda I-A)^{-1}$ is positive for every $\lambda\in\rho(A)$ with $\lambda\ge \lambda_1$, due to \cite[Chapter 7]{CleHeiAngDuiPag1987}. As $\lambda_0 I-A\colon \mathcal{D}(A)\to X$ is positive and local, Corollary \ref{posresolvent} yields that $T(t)$ is local for every $t\in [0,\infty)$.
\end{proof}

Merely as illustration, we present an example of an ordered Banach space satisfying the conditions of Corollary \ref{posresolvent}, on which there exists a non-trivial multiplication operator $A$ which generates a $\mathrm{C}_0$-semigroup.

\begin{example}
Consider the locally compact Hausdorff space $[0,1)$ and 
\[X=\left\{x\in\mathrm{C}_0([0,1));\, x|_{[0,\frac{1}{2}]}\in \mbox{Pol}^2[0,{\textstyle\frac{1}{2}}]\right\}, \]
where $\mathrm{C}_0([0,1))$ is defined in Remark \ref{rem:mult_opt}(b) and $\mbox{Pol}^2[a, b]$ is the space of  polynomial functions of at most degree 2 on $[a, b]$. As in the proof of \cite[Example 3.5]{KalGaa2006}, it can be verified that $X$ is order dense in $\mathrm{C}_0([0, 1))$. Thus, $X$ is a pre-Riesz space and the embedding map $i\colon X\to \mathrm{C}_0([0, 1))$ is a Riesz* homomorphism. 

Moreover, $X$
is a closed subspace of $(\mathrm{C}_0([0,1)), \left\|\cdot\right\|_\infty)$. 
Indeed, let $(x_n)$ be a sequence in $X$ and $x\in\mathrm{C}_0([0, 1))$ be such that 
$\left\|x_n-x\right\|_\infty\rightarrow0$. 
Then $(x_n |_{[0,\frac{1}{2}]})$ is a sequence in $\mbox{Pol}^2[0,\frac{1}{2}]$ and 
\[\left\|x_n |_{[0,\frac{1}{2}]}-x|_{[0,\frac{1}{2}]}\right\|\le \left\|x_n-x\right\|_\infty\rightarrow0.\]
Since $(\mbox{Pol}^2[0,\frac{1}{2}], \left\|\cdot\right\|_\infty)$ is finite dimensional, it is closed in $(\mathrm{C}[0, \frac{1}{2}], \left\|\cdot\right\|_\infty)$. Thus, it follows that $x|_{[0,\textstyle\frac{1}{2}]}\in \mbox{Pol}^2[0,\frac{1}{2}]$ and hence $x\in X$.

Since the cone in $X$ is closed and generating, $X$ is an ordered Banach space.

Let $q\in C([0,1))$ be bounded above and constant on $[0, \frac{1}{2}]$. As in Example \ref{loc-exa} (ii), 
 let $Ax=qx$ for  $x\in \mathcal{D}(A)=\{x\in X;\, s\mapsto q(s)x(s)\in X\}$. The elements of $X$ that vanish on an interval $[1-\varepsilon,1)$ for some $\varepsilon>0$ are norm dense in $X$ and they are in $\mathcal{D}(A)$, hence $\mathcal{D}(A)$ is norm dense in $X$. Also, $A$ is closed. Indeed, let $x_n\in \mathcal{D}(A)$ and $x,y\in X$ be such that $\left\|x_n-x\right\|_\infty\rightarrow0$ and $\left\|Ax_n-y\right\|_\infty\rightarrow0$. 
 Then for every $t\in [0,1)$ we have that $x_n(t)\to x(t)$, hence $q(t)x_n(t)\to q(t)x(t)$, and $(qx_n)(t)\to y(t)$, and therefore $q(t)x(t)=y(t)$. Hence $x\in\mathcal{D}(A)$ and $Ax=y$. 

Next we show that $A\colon X\supseteq\mathcal{D}(A)\to X$ is local. Let $x\in\mathcal{D}(A)$ and $y\in X$ be such that $x\perp y$. Since $X$ is order dense in $\mathrm{C}_0([0,1))$, it follows that $x\perp y$ in $\mathrm{C}_0([0,1))$. Therefore, for every $t\in [0,1)$ we have $x(t)=0$ or $y(t)=0$, hence $q(t)x(t)=0$ or $y(t)=0$, which yields that $Ax\perp y$ in $\mathrm{C}_0([0,1))$ and hence in $X$.

Let $\lambda_0\in [0,\infty)$ with $\lambda_0> \sup_s q(s)$. Then we have that $\lambda_0 I-A$ is positive and local and for every $\lambda\ge \lambda_0$ we have that $(\lambda I-A)^{-1}$ is positive. Now $A$ satisfies all conditions of Corollary \ref{posresolvent}, provided $A$ generates a $\mathrm{C}_0$-semigroup $T$. This is in fact the case, namely $T$ is given by  $(T(t)x)(s)=e^{q(s)t}x(s)$, $s\in [0,1]$, $t\in [0,\infty)$.  Clearly, $T(t)$ is local for every $t\in [0,\infty)$.
\end{example}

The conditions in our analysis that the norm is semimonotone and the space is norm complete appear fairly weak, but exclude in fact many interesting examples such as $\mathrm{C}^k$-spaces and Sobolev spaces. A general theory of disjointness preserving $\mathrm{C}_0$-semigroups on such spaces will be an interesting topic of further research.

\begin{acknowledgement*}
Feng Zhang is supported by a PhD scholarship of the China Scholarship Council. 
\end{acknowledgement*}

\bibliographystyle{plain}
\bibliography{Pre-Riesz_Sp_1_Lit}

\end{document}